\documentclass{commat}

\usepackage{stackengine}
\usepackage{scalerel}

\title[Notes on geodesics of cotangent bundle]{Note on geodesics of cotangent bundle with Berger-type deformed Sasaki metric over standard K\"{a}hler manifold}

\author{Abderrahim Zagane}

\authorinfo{%
    Relizane University, Algeria}{%
    Zaganeabr2018@gmail.com
    }

\abstract{%
    In this paper, first, we introduce the Berger-type deformed Sasaki metric on the cotangent bundle $T^{\ast}M$ over a standard K\"{a}hler manifold $(M^{2m}, J, g)$ and investigate the Levi-Civita connection of this metric. Secondly, we present the unit cotangent bundle equipped with Berger-type deformed Sasaki metric, and we investigate the Levi-Civita connection. Finally, we study the geodesics on the cotangent bundle and the unit cotangent bundle concerning the Berger-type deformed Sasaki metric.
    }



\VOLUME{32}
\YEAR{2024}
\NUMBER{1}
\firstpage{241}
\DOI{https://doi.org/10.46298/cm.11025}

\begin{document}

\section{Introduction}

One can define natural Riemannian metrics on the cotangent bundle of a Riemannian manifold. Their construction makes use of the Levi-Civita connection. Among them, the so-called Sasaki metric is of particular interest. That is why A.A. Salimov and F. Agca \ have studied the geometry of a cotangent bundle equipped with the Sasaki metric \cite{S.A1, S.A2}, A.A. Salimov and  F. Ocak \cite{S.O}. The rigidity of Sasaki metric has incited some researchers to construct and study other metrics on the cotangent bundle. This is the reason why some authors have attempted to search for different metrics on the cotangent bundle, which are different deformations of the Sasaki metric. In this direction, some authors defined and studied some metrics, which are called Cheeger-Gromoll metric \cite{A.S} or $g$-Natural Metrics \cite{Agc,Z.Z1} or new metric in the cotangent bundle \cite{O.K,Oca} or a new class of metrics on the cotangent bundle \cite{Zag1,Z.Z2}. In another direction, A. Zagane has introduced the notion of Berger-type deformed Sasaki metric on the cotangent bundle over anti-paraK\"{a}hler manifold \cite{Zag4,Zag5,Zag9}. For deformations of the Sasaki metric or Cheeger-Gromoll metric, we also refer to \cite{G.A, Zag6, Zag12, Z.G}.

The main idea in this paper, firstly, we introduce the Berger-type deformed Sasaki metric on the cotangent bundle $T^{\ast}M$ over a standard K\"{a}hler manifolds manifold $(M^{2m}, J, g)$ and we investigate the formulas relating to its Levi-Civita connection (Theorem \ref{th_01}). Secondly, we present the unit cotangent bundle equipped with the Berger-type deformed Sasaki metric, and we establish the formulas relating to the Levi-Civita connection of this metric (Theorem \ref{th_02}). In the last section, we study the geodesics on the cotangent bundle (Theorem \ref{th_03}, Corollary \ref{co_01}, Corollary \ref{co_02} and Theorem \ref{th_04}) and on unit cotangent bundle (Theorem \ref{th_05}, Theorem \ref{th_06}, Theorem \ref{th_07} and Theorem \ref{th_08}).

\section{Preliminaries}

Let $(M^{m},g)$ be an $m$-dimensional Riemannian manifold, $T^{\ast}M$ be its cotangent bundle and $\pi:T^{\ast}M\rightarrow M$ the natural projection. A local chart $(U,x^{i})_{i=\overline{1,m}}$ on $M$ induces a local chart $(\pi^{-1}(U),x^{i},x^{\bar{i}}=p_{i})_{\bar{i}=\overline{m+1,2m}}$ on $T^{\ast}M$, where $p_{i}$ is the component of covector $p$ in each cotangent space $T_{x}^{\ast}M$, $x\in U$ with respect to the natural coframe $\{dx^{i}\}$, denote  by $\partial_{i}=\frac{\partial}{\partial x^{i}}$ and $\partial_{\bar{i}}=\frac{\partial}{\partial x^{\bar{i}}}$. Let $C^{\infty}(M)$ (resp. $C^{\infty}(T^{\ast}M)$) be the ring of real-valued $C^{\infty}$ functions on $M$(resp. $T^{\ast}M$) and $\Im^{r}_{s}(M) $ (resp. $\Im^{r}_{s}(T^{\ast}M)$) be the module over $C^{\infty}(M)$ (resp. $C^{\infty}(T^{\ast}M)$) of $C^{\infty}$ tensor fields of type $(r, s)$. Denote by $\Gamma_{ij}^{k}$ the Christoffel symbols of $g$ and by $\nabla$ the Levi-Civita connection of $g$.

The Levi Civita connection $\nabla$ defines a direct sum decomposition
\begin{eqnarray}\label{D.S}
TT^{\ast}M=VT^{\ast}M\oplus HT^{\ast}M
\end{eqnarray}
of the tangent bundle to $T^{\ast}M$ at any $(x,p)\in T^{\ast}M$ into vertical subspace
\begin{eqnarray}\label{V.D}
V_{(x,p)}T^{\ast}M=\ker(d\pi_{(x,p)})=\{\omega_{i}\partial_{\bar{i}}|_{(x,p)},\, \omega_{i}\in\mathbb{R}\},
\end{eqnarray}
and the horizontal subspace
\begin{eqnarray}\label{H.D}
H_{(x,p)}T^{\ast}M=\{X^{i}\partial_{i}|_{(x,p)}+X^{i}p_{a}\Gamma_{hi}^{a}\partial_{\bar{h}}|_{(x,p)},\, X^{i}\in \mathbb{R}\}.
\end{eqnarray}

Note that the map $X \rightarrow{}^{H}\!X=X^{i}\partial_{i}|_{(x,p)}+X^{i}p_{a}\Gamma_{hi}^{a}\partial_{\bar{h}}|_{(x,p)}$ is an isomorphism between the vector spaces $T_{x}M$ and $H_{(x,p)}T^{\ast}M$. 

Similarly, the map $\omega \rightarrow {}^{V}\!\omega=\omega_{i}\partial_{\bar{i}}|_{(x,p)}$ is an isomorphism between the vector spaces $T_{x}^{\ast}M$ and $V_{(x,p)}T^{\ast}M$. Obviously, each tangent vector $Z \in T_{(x,p)}T^{\ast}M$ can be written in the form $ Z = {}^{H}\!X + {}^{V}\!\omega$,
where $X \in T_{x}M$ and $\omega \in T_{x}^{\ast}M$ are uniquely determined.

Let $X=X^{i}\partial_{i}$ and $\omega=\omega_{i} dx^{i}$ be local expressions in  $(U,x^{i})_{i=\overline{1,m}}$, of a vector and covector ($1$-form) field $X\in\Im^{1}_{0}(M)$ and $\omega\in\Im^{0}_{1}(M)$, respectively. Then the horizontal lift ${}^{H}\!X \in \Im^{1}_{0}(T^{\ast}M)$ of $X\in\Im^{1}_{0}(M)$ and the vertical lift ${}^{V}\!\omega\in \Im^{1}_{0}(T^{\ast}M)$ of  $\omega\in\Im^{0}_{1}(M)$ are defined, respectively by
\begin{eqnarray}\label{Hori}
{}^{H}\!X&=&X^{i}\partial_{i}+ p_{h}\Gamma_{ij}^{h}X^{j}\partial_{\bar{i}},\\\label{Vert}
{}^{V}\!\omega&=& \omega_{i}\partial_{\bar{i}},
\end{eqnarray}
with respect to the natural frame $\{\partial_{i}, \partial_{\bar{i}}\}$, (see \cite{Y.I} for more details).

From $(\ref{Hori})$ and $(\ref{Vert})$ we see that ${}^{H}\!(\partial_{i})$ and $ {}^{V}\!(dx^{i})$ have respectively local expressions of the form 
\begin{eqnarray}\label{Hori2}
{}^{H}\!(\partial_{i})&=&\partial_{i}+ p_{a}\Gamma_{hi}^{a}\partial_{\bar{h}},\\\label{Vert2}
{}^{V}\!(dx^{i})&=&\partial_{\bar{i}}.
\end{eqnarray}
The set of vector fields $\{{}^{H}\!(\partial_{i})\}$ on $\pi^{-1}(U)$ define a local frame for $HT^{\ast}M$ over $\pi^{-1}(U)$ and the set of vector fields $\{{}^{V}\!(dx^{i})\}$ on $\pi^{-1}(U)$ define a local frame for $VT^{\ast}M$ over $\pi^{-1}(U)$. The set $\{{}^{H}\!(\partial_{i}),{}^{V}\!(dx^{i})\}$  define a local frame on $T^{\ast}M$, adapted to the direct sum decomposition~$(\ref{D.S})$.

In particular, we have the vertical spray ${}^{V}\!p$ on $T^{\ast}M$ defined by
\begin{eqnarray}\label{Liouville vector}
{}^{V}\!p= p_{i}{}^{V}\!(dx^{i}) =p_{i}\partial_{\bar{i}},
\end{eqnarray}
${}^{V}\!p$ is also called the canonical or Liouville vector field on $T^{\ast}M$.

\begin{lemma}[\cite{Y.I}] \label{lem_00}
Let $(M^{m}, g)$ be a Riemannian manifold, $\nabla$ be the Levi-Civita connection, and $R$ be the Riemannian curvature tensor. Then the Lie bracket of the cotangent bundle $T^{\ast}M$ of $M$ satisfies the following
\begin{enumerate}
\item[(1)] $[{}^{V}\!\omega,{}^{V}\!\theta]=0,$
\item[(2)] $[{}^{H}\!X,{}^{V}\!\theta]={}^{V}\!(\nabla_{X}\theta),$
\item[(3)] $[{}^{H}\!X,{}^{H}\!Y]={}^{H}\![X,Y]+{}^{V}\!(pR(X,Y)),$
\end{enumerate}
for all $X,Y\in\Im^{1}_{0}(M)$ and $\omega,\theta\in\Im^{0}_{1}(M)$, such that $pR(X,Y)=p_{a}R^{a}_{ijk}X^{i}Y^{j}\,dx^{k}$, where $R^{a}_{ijk}$ are local components of $R$ on $(M^{m},g)$.
\end{lemma}

Let $(M^{m}, g)$ be a Riemannian manifold, we define the map
\begin{eqnarray*}\begin{array}{ccc}
\Im^{0}_{1}(M)&\rightarrow&\Im^{1}_{0}(M)\\
\omega&\mapsto&\widetilde{\omega}
\end{array}\quad\textit{by}\quad
\begin{array}{ccc}
g(\widetilde{\omega}, X)&=&\omega(X),
\end{array}
\end{eqnarray*} 
for all $X\in\Im^{1}_{0}(M)$. Locally if $\omega=\omega_{i}dx^{i}\in\Im^{0}_{1}(M)$, we have 
\begin{eqnarray}\label{tilde_omega}
\widetilde{\omega}=g^{ij}\omega_{i}\partial_{j},
\end{eqnarray}
where $(g^{ij})$ is the inverse matrix of the matrix $(g_{ij})$.

The scalar product $g^{-1}=(g^{ij})$ is defined on the cotangent space $T_{x}^{\ast}M$ by 
\begin{eqnarray*}
g^{-1}(\omega, \theta) =g(\widetilde{\omega},\tilde{\theta}) =g^{ij}\omega_{i}\theta_{j},
\end{eqnarray*}
for all $x\in M$ and $\omega,\theta\in\Im^{0}_{1}(M)$. In this case we have $\widetilde{\omega}=g^{-1}\circ\omega$.

We also define the map 
\begin{eqnarray*}\begin{array}{ccc}
\Im^{1}_{0}(M)&\rightarrow&\Im^{0}_{1}(M)\\
X&\mapsto&\widetilde{X}
\end{array}\quad\textit{by}\quad
\begin{array}{ccc}
\widetilde{X}(Y)&=&g(X, Y),
\end{array}
\end{eqnarray*} 
for all $Y\in\Im^{1}_{0}(M)$. Locally if $X=X^{i}\partial_{i}\in\Im^{1}_{0}(M)$, we have 
\begin{eqnarray}\label{tilde_X}
\widetilde{X}=g_{ij}X^{i}dx^{j},
\end{eqnarray}
we also write $\widetilde{X}=g\circ X$.

\begin{lemma}[\cite{Zag5}] \label{lem_01}
Let $(M, g)$ be a Riemannian manifold, then we have the following:
\begin{eqnarray}\label{eq_00}
\widetilde{\widetilde{\omega}}=\omega &,&  \widetilde{\widetilde{X}}=X,\\\label{eq_01}
g^{-1}(\omega, \theta J)&=&g(J\widetilde{\omega}, \tilde{\theta}),\\\label{eq_02}
\nabla_{X}\widetilde{\omega}&=&\widetilde{\nabla_{X}\omega},\\\label{eq_03}
Xg^{-1}(\omega,\theta)&=&g^{-1}(\nabla_{X}\omega,\theta)+g^{-1}(\omega,\nabla_{X}\theta),\\\label{eq_04}
\stackon[-7pt]{$\omega R(X,Y)$}{\vstretch{1.5}{\hstretch{7.2}{\widetilde{\phantom{\;}}}}}&=&R(Y,X)\widetilde{\omega},
\end{eqnarray}  
for all $X, Y\in\Im^{1}_{0}(M)$, $\omega,\theta\in\Im^{0}_{1}(M)$ and $J\in\Im^{1}_{1}(M)$, where $\nabla$ is the Levi-Civita connection of $(M,g)$.
\end{lemma}

\section{Berger-type deformed Sasaki metric}

Let $M^{r}$ be an $r$-dimensional differentiable manifold. An almost complex
structure $J$ on $M$ is a $(1,1)$-tensor field on $M$ such that $
J^{2}=-I$, ($I$ is the $(1,1)$-identity tensor field on $M.$ The pair $
(M^{r},J)$ is called an almost complex manifold. Since every almost complex manifold is even dimensional, We will take $r=2m$. Also, note that every
complex manifold (Topological space endowed with a holomorphic atlas)
carries a natural almost complex structure \cite{K.N}. The integrability of a complex structure $J$ on $M$ is equivalent to
the vanishing of the Nijenhuis tensor $N_{J}$: 
\begin{equation}
N_{J}(X,Y)=\left[ JX,JY\right] -J\left[ JX,Y\right] -J\left[ X,JY\right] -
\left[ X,Y\right]  \label{Nij-tensor}
\end{equation}
for all vector fields $X,Y$ on $M$. 

On an almost complex manifold $(M^{2m},J)$, a Hermitian metric is a
Riemannian metric $g$ on $M$ such that 
\begin{eqnarray}\label{Herm-m}
g(JX,Y)=-g(X,JY)\Leftrightarrow g(JX, JY)= g(X, Y),
\end{eqnarray}
or  from $(\ref{eq_01})$ equivalently 
\begin{eqnarray}\label{Herm-m1}
g^{-1}(\omega J, \theta )=-g^{-1}(\omega,\theta J) \Leftrightarrow g^{-1}(\omega J, \theta J )= g^{-1}(\omega,\theta),
\end{eqnarray}
for all $X, Y \in\Im^{1}_{0}(M)$ and $\omega, \theta \in\Im^{0}_{1}(M)$. 

The almost complex manifold $(M^{2m},J)$ having the Hermitian metric $g$ is called an almost Hermitian
manifold. Let $(M^{2m},J,g)$ be an almost Hermitian manifold. We define the
fundamental or K\"{a}hler $2$-form $\Omega $ on $M$ by 
\begin{equation}
\Omega (X,Y)=g(X,JY),  \label{fund-form}
\end{equation}
for any vector fields $X$ and $Y$ on $M$. A Hermitian metric $g$ on an
almost Hermitian manifold $M^{2m}$ is called a standard K\"{a}hler metric if the
fundamental $2$-form $\Omega $ is closed, i.e., $d\Omega =0$. In this case, the triple $(M^{2m},J,g)$ is called an almost standard K\"{a}hler manifold. If the
almost complex structure is integrable, then the triple $(M^{2m},J,g)$ is
called a standard K\"{a}hler manifold. Moreover,  the following conditions are
equivalent
\begin{enumerate}
\item $\nabla J=0$, 
\item $\nabla \Omega =0$,
\item $N_{J}=0$ and $d\Omega =0,$ 
\end{enumerate}
where $\nabla $ is the Levi-Civita connection of $g$ \cite{K.N}.

As a result, the almost Hermitian manifold $(M^{2m},J,g)$ is a standard K\"{a}hler
manifold if and only if $\nabla J=0$. Using the formula 
\begin{equation}\label{eq_05} 
\omega(\nabla_{X} J)= \nabla_{X} (\omega J)- (\nabla_{X}\omega) J.	
\end{equation}
Also, the almost Hermitian manifold $(M^{2m},J,g)$ is a standard K\"{a}hler
manifold if and only if
\begin{equation}\label{eq_06} 
\nabla_{X} (\omega J) = (\nabla_{X}\omega) J.	
\end{equation}
for all $X\in\Im^{1}_{0}(M)$, $\omega\in\Im^{0}_{1}(M)$.
The Riemannian curvature tensor $R$ of
a standard K\"{a}hler manifold possess the following properties: 
\begin{equation}\label{eq_07}
\left\{ 
\begin{array}{ll}
R(Y,Z)J & =JR(Y,Z), \\ 
R(JY,JZ) & =R(Y,Z), \\ 
R(JY,Z) & =-R(Y,JZ),
\end{array}%
\right.
\end{equation}
for all $Y,Z\in\Im^{1}_{0}(M)$.

\begin{lemma}
Let $(M^{2m}, J,g)$ be an almost  Hermitian manifold. We have the following:
\begin{eqnarray}\label{eq_08}
\widetilde{\omega J}&=&-J\widetilde{\omega},
\end{eqnarray}
for any $\omega\in\Im^{0}_{1}(M)$. 	
\end{lemma}		

\begin{definition}\label{def_01}
Let $(M^{2m}, J,g)$ be an almost  Hermitian manifold and $T^{\ast}M$ be its cotangent bundle. A fiber-wise Berger-type deformation of the Sasaki metric noted ${}^{BS}\!g$ is defined on $T^{\ast}M$ by 
\begin{eqnarray*}
{}^{BS}\!g({}^{H}\!X,{}^{H}\!Y) &=&g(X,Y),\\
{}^{BS}\!g({}^{H}\!X,{}^{V}\!\theta) &=& 0,\\
{}^{BS}\!g({}^{V}\!\omega,{}^{V}\!\theta) &=& g^{-1}(\omega,\theta)+\delta^{2} g^{-1}(\omega, p J)g^{-1}(\theta, p J),
\end{eqnarray*}
for all  $X, Y\in\Im^{1}_{0}(M)$, $\omega,\theta\in\Im^{0}_{1}(M)$, where $\delta$ is some constant \cite{Zag4},\cite{Zag5},\cite{Zag9}, for version tangent bundle see \cite{Yam}.
\end{definition}

In the following, we put $\lambda =1+\delta ^{2}r^{2}$ and $r^{2}=g^{-1}(p,p)=|p|^{2}$. where $|.|$ denote the norm with respect to $g^{-1}$.

\begin{lemma}\label{lem_02}
Let $(M^{2m}, J, g)$ be a standard K\"{a}hler manifold and $f:\mathbb{R}\rightarrow\mathbb{R}$ a smooth function, we have the following:
\begin{enumerate}
\item ${}^{H}\!X(f(r^{2}))=0,$
\item ${}^{V}\!\theta(f(r^{2}))=2f^{\prime}(r^{2})g^{-1}(\theta,p),$	
\item ${}^{H}\!Xg^{-1}(\omega,p)=g^{-1}(\nabla_{X}\omega,p),$
\item ${}^{V}\! \theta g^{-1}(\omega,p)=g^{-1}(\omega, \theta),$
\item ${}^{H}\!Xg^{-1}(\omega,p J)=g^{-1}(\nabla_{X}\omega,p J),$
\item ${}^{V}\! \theta g^{-1}(\omega,p J)=g^{-1}(\omega, \theta J),$
\end{enumerate}
for all $ X \in \Im^{1}_{0}(M)$ and $\omega, \theta \in\Im^{0}_{1}(M)$, where $r^{2}=g^{-1}(p,p)$, see \cite{Zag1},\cite{Zag5}.
\end{lemma}

\begin{lemma}\label{lem_03}
Let $(M^{2m}, J, g)$ be a standard K\"{a}hler manifold and $T^{\ast}M$ its cotangent bundle equipped with the Berger-type deformed Sasaki metric ${}^{BS}\!g$, we have the following:
\begin{eqnarray*}
(1)\;\; {}^{H}\!X{}^{BS}\!g({}^{V}\!\omega,{}^{V}\!\theta)&=&{}^{BS}\!g({}^{V}\!(\nabla_{X}\omega),{}^{V}\!\theta)+{}^{BS}\!g({}^{V}\!\omega,{}^{V}\!(\nabla_{X}\theta)),\\
(2)\;\;\; {}^{V}\!\eta {}^{BS}\!g({}^{V}\!\omega,{}^{V}\!\theta)&=&\delta^{2}g^{-1}(\omega,\eta J)g^{-1}(\theta,p J)+\delta^{2}g^{-1}(\omega,p J)g^{-1}(\theta, \eta J),
\end{eqnarray*}
for all $X\in\Im^{1}_{0}(M)$ and $\omega, \theta, \eta\in\Im^{0}_{1}(M)$, see as well \cite{Zag5}
\end{lemma}

We shall calculate the Levi-Civita connection $ {}^{BS}\nabla$ of $T^{\ast}M$ with Berger-type deformed Sasaki metric ${}^{BS}\!g$. The Koszul formula characterizes this connection:
\begin{eqnarray}\label{F.K}
2{}^{BS}\!g( {}^{BS}\nabla_{U}V,W)=U{}^{BS}\!g(V,W)+V{}^{BS}\!g(W,U)-W{}^{BS}\!g(U,V)\qquad\qquad\notag\\
+{}^{BS}\!g(W, [U,V])+{}^{BS}\!g(V, [W,U])-{}^{BS}\!g(U, [V,W]),\;\;
\end{eqnarray}
for all $U,V,W\in\Im^{1}_{0}(T^{\ast}M)$.

\begin{theorem}\label{th_01}
Let $(M^{2m}, J, g)$ be a standard K\"{a}hler manifold and $T^{\ast}M$ its cotangent bundle equipped with the Berger-type deformed Sasaki metric ${}^{BS}\!g$, then we have the following formulas:
\begin{eqnarray*}
(i)\;\;{}^{BS}\nabla_{{}^{H}\!X}{}^{H}\!Y&=&{}^{H}\!(\nabla_{X}Y)+\dfrac{1}{2}{}^{V}\!(pR(X,Y)),\\
(ii)\;\;{}^{BS}\nabla_{{}^{H}\!X}{}^{V}\!\theta&=&{}^{V}\!(\nabla_{X}\theta)+\dfrac{1}{2} \big({}^{H}\!(R(\widetilde{p},\widetilde{\theta})X)- \delta^{2}g^{-1}(\theta,p J){}^{H}\!(R( \widetilde{p},J \widetilde{p})X)\big),\\
(iii)\;\;{}^{BS}\nabla_{{}^{V}\!\omega}{}^{H}\!Y&=&\dfrac{1}{2} \big({}^{H}\!(R(\widetilde{p},\widetilde{\omega})Y)- \delta^{2}g^{-1}(\omega,p J){}^{H}\!(R(\widetilde{p},J \widetilde{p})Y)\big),\\
(iv)\;\;{}^{BS}\nabla_{{}^{V}\!\omega}{}^{V}\!\theta&=&\delta^{2}\big(g^{-1}(\omega, p J){}^{V}\!(\theta J)+g^{-1}(\theta, p J){}^{V}\!(\omega J)\big)\\
&&-\frac{\delta^{4}}{\lambda}\big(g^{-1}(\omega, p J)g^{-1}(\theta, p)+g^{-1}(\omega, p)g^{-1}(\theta, p J)\big){}^{V}\!(p J),
\end{eqnarray*}
for all  $X,Y\in \Im^{1}_{0}(M)$ and $\omega,\theta\in\Im^{0}_{1}(M)$, where $\nabla$ is the Levi-Civita connection of $(M^{2m}, J, g)$ and $R$ is its curvature tensor, (for anti-paraK\"{a}hler manifold, see \cite{Zag5}).
\end{theorem}

\begin{proof}
The proof of Theorem $\ref{th_01}$ follows directly from Kozul formula $(\ref{F.K})$, Lemma $\ref{lem_00}$, Definition $\ref{def_01}$ and Lemma $\ref{lem_03}$.

$(1)$\;Direct calculations give
\begin{eqnarray*}
2{}^{BS}\!g({}^{BS}\nabla_{{}^{H}\!X}{}^{H}\!Y,{}^{H}\!Z)&=&{}^{H}\!X{}^{BS}\!g({}^{H}\!Y,{}^{H}\!Z)+{}^{H}\!Y{}^{BS}\!g({}^{H}\!Z,{}^{H}\!X)-{}^{H}\!Z{}^{BS}\!g({}^{H}\!X,{}^{H}\!Y)\\
&&+{}^{BS}\!g({}^{H}\!Z,[{}^{H}\!X,{}^{H}\!Y])+{}^{BS}\!g({}^{H}\!Y,[{}^{H}\!Z,{}^{H}\!X])-{}^{BS}\!g({}^{H}\!X,[{}^{H}\!Y,{}^{H}\!Z])\\
&=&Xg(Y,Z)+Yg(Z,X)-Zg(X,Y)+g(Z,[X,Y])\\
&&+g(Y,[Z,X])-g(X,[Y,Z])\\
&=&2g(\nabla_{X}Y,Z)\\
&=& 2{}^{BS}\!g({}^{H}\!(\nabla_{X}Y),{}^{H}\!Z),\\
{}^{BS}\!g({}^{BS}\nabla_{{}^{H}\!X}{}^{H}\!Y,{}^{V}\!\eta)&=&{}^{H}\!X{}^{BS}\!g({}^{H}\!Y,{}^{V}\!\eta)+{}^{H}\!Y{}^{BS}\!g({}^{V}\!\eta,{}^{H}\!X)-{}^{V}\!\eta {}^{BS}\!g({}^{H}\!X,{}^{H}\!Y)\\
&&+{}^{BS}\!g({}^{V}\!\eta,[{}^{H}\!X,{}^{H}\!Y])+{}^{BS}\!g({}^{H}\!Y,[{}^{V}\!\eta,{}^{H}\!X])-{}^{BS}\!g({}^{H}\!X,[{}^{H}\!Y,{}^{V}\!\eta])\\
&=&{}^{BS}\!g({}^{V}\!\eta,[{}^{H}\!X,{}^{H}\!Y])\\
&=&{}^{BS}\!g({}^{V}\!(pR(X,Y)),{}^{V}\!\eta),
\end{eqnarray*}
Thus, we find
\begin{eqnarray*}
{}^{BS}\nabla_{{}^{H}\!X}{}^{H}\!Y={}^{H}\!(\nabla_{X}Y)+\dfrac{1}{2}{}^{V}\!(pR(X,Y)).
\end{eqnarray*}

$(2)$ In a similar way,
\begin{eqnarray*}
2{}^{BS}\!g({}^{BS}\nabla_{{}^{H}\!X}{}^{V}\!\theta,{}^{H}\!Z)&=&{}^{H}\!X{}^{BS}\!g({}^{V}\!\theta,{}^{H}\!Z)+{}^{V}\!\theta {}^{BS}\!g({}^{H}\!Z,{}^{H}\!X)-{}^{H}\!Z{}^{BS}\!g({}^{H}\!X,{}^{V}\!\theta)\\
&&+{}^{BS}\!g({}^{H}\!Z,[{}^{H}\!X,{}^{V}\!\theta])+{}^{BS}\!g({}^{V}\!\theta,[{}^{H}\!Z,{}^{H}\!X])-{}^{BS}\!g({}^{H}\!X,[{}^{V}\!\theta,{}^{H}\!Z])\\
&=&{}^{BS}\!g({}^{V}\!\theta,[{}^{H}\!Z,{}^{H}\!X])\\
&=& {}^{BS}\!g({}^{V}\!(pR(Z,X)),{}^{V}\!\theta)\\
&=&g^{-1}(pR(Z,X),\theta)+\delta^{2}g^{-1}(pR(Z,X),p J)g^{-1}(\theta,p J),
\end{eqnarray*}
From $(\ref{eq_04})$, we have
\begin{eqnarray*}
g^{-1}(pR(Z,X),\theta)&=&g(\stackon[-7pt]{pR(Z,X)}{\vstretch{1.5}{\hstretch{7.2}{\widetilde{\phantom{\;}}}}},\tilde{\theta})=g(R(X,Z)\tilde{p},\tilde{\theta})=g(R(\tilde{p},\tilde{\theta} )X, Z)\\
&=&{}^{BS}\!g({}^{H}\!(R(\tilde{p},\tilde{\theta})X),{}^{H}\!Z).
\end{eqnarray*}
On the other hand, using  $(\ref{eq_01})$, $(\ref{eq_04})$ and $(\ref{eq_07})$, we have 
\begin{eqnarray*}
g^{-1}(pR(Z,X),p J)&=&g(J(\stackon[-7pt]{pR(Z,X)}{\vstretch{1.5}{\hstretch{7.2}{\widetilde{\phantom{\;}}}}}),\tilde{p})=g(J R(X,Z)\tilde{p},\tilde{p})\notag\\
&=&g( R(X, Z)J \tilde{p},\tilde{p})=g( R(J \tilde{p},\tilde{p})X, Z)\\
&=&{}^{BS}\!g( {}^{H}\!(R(J \tilde{p},\tilde{p})X), {}^{H}\!Z),
\end{eqnarray*}
then,
\begin{eqnarray*}
2{}^{BS}\!g({}^{BS}\nabla_{{}^{H}\!X}{}^{V}\!\theta,{}^{H}\!Z)={}^{BS}\!g({}^{H}\!(R(\tilde{p},\tilde{\theta})X)-\delta^{2}g^{-1}(\theta,p J){}^{H}\!(R(\tilde{p},J \tilde{p})X), {}^{H}\!Z),
\end{eqnarray*}	
and also with direct calculations, we obtain
\begin{eqnarray*}
2{}^{BS}\!g({}^{BS}\nabla_{{}^{H}\!X}{}^{V}\!\theta,{}^{V}\!\eta)&=&{}^{H}\!X{}^{BS}\!g({}^{V}\!\theta,{}^{V}\!\eta)+{}^{V}\!\theta {}^{BS}\!g({}^{V}\!\eta,{}^{H}\!X)-{}^{V}\!\eta {}^{BS}\!g({}^{H}\!X,{}^{V}\!\theta)\\
&&+{}^{BS}\!g({}^{V}\!\eta,[{}^{H}\!X,{}^{V}\!\theta])+{}^{BS}\!g({}^{V}\!\theta,[{}^{V}\!\eta,{}^{H}\!X])-{}^{BS}\!g({}^{H}\!X,[{}^{V}\!\theta,{}^{V}\!\eta])\\
&=&{}^{H}\!X{}^{BS}\!g({}^{V}\!\theta,{}^{V}\!\eta)+{}^{BS}\!g({}^{V}\!\eta,[{}^{H}\!X,{}^{V}\!\theta])+{}^{BS}\!g({}^{V}\!\theta,[{}^{V}\!\eta,{}^{H}\!X]).
\end{eqnarray*}
Using the first formula of Lemma $\ref{lem_03}$ we have:
\begin{eqnarray*}
2{}^{BS}\!g({}^{BS}\nabla_{{}^{H}\!X}{}^{V}\!\theta,{}^{V}\!\eta)&=&{}^{BS}\!g({}^{V}\!(\nabla_{X}\theta),{}^{V}\!\eta)+{}^{BS}\!g({}^{V}\!\theta,{}^{V}\!(\nabla_{X}\eta))\\
&&+{}^{BS}\!g({}^{V}\!\eta,{}^{V}\!(\nabla_{X}\theta))-{}^{BS}\!g({}^{V}\!\theta,{}^{V}\!(\nabla_{X}\eta))\\
&=&2{}^{BS}\!g({}^{V}\!(\nabla_{X}\theta),{}^{V}\!\eta).
\end{eqnarray*}
Which gives the formula 
\begin{eqnarray*}
{}^{BS}\nabla_{{}^{H}\!X}{}^{V}\!\theta={}^{V}\!(\nabla_{X}\theta)+\dfrac{1}{2} \big({}^{H}\!(R(\tilde{p},\tilde{\theta})X)- g^{-1}(\theta,p J){}^{H}\!(R(\tilde{p},J \tilde{p})X)\big).
\end{eqnarray*}
Similar calculations obtain the other formulas.
\end{proof}

As a consequence of  Theorem \ref{th_01}, we get the following Lemma.

\begin{lemma}\label{lem_04}
Let $(M^{n}, J, g)$ be a standard K\"{a}hler manifold and $(T^{\ast}M,{}^{BS}\!g)$ its cotangent bundle equipped with the Berger-type deformed Sasaki metric, then
\begin{eqnarray*}
{}^{BS}\nabla_{{}^{H}\!X}{}^{V}\!p&=&0,\\
{}^{BS}\nabla_{{}^{V}\!p}{}^{H}\!X&=&0,\\
{}^{BS}\nabla_{{}^{V}\!\omega}{}^{V}\!p&=&{}^{V}\!\omega+\frac{\delta^{2}}{\lambda}g^{-1}(\omega, pJ){}^{V}\!(p J),\\
{}^{BS}\nabla_{{}^{V}\!p}{}^{V}\!\omega&=&\frac{\delta^{2}}{\lambda}g^{-1}(\omega, pJ){}^{V}\!(p J),\\
{}^{BS}\nabla_{{}^{V}\!p}{}^{V}\!p&=&{}^{V}\!p,
\end{eqnarray*}
for all $X\in \Im^{1}_{0}(M)$ and $\omega\in\Im^{0}_{1}(M)$.
\end{lemma}


\section{Unit cotangent bundle with Berger-type deformed Sasaki metric}

The unit cotangent (sphere) bundle over a standard K\"{a}hler manifold  $(M^{n}, J, g)$, is the hyper-surface 
\begin{eqnarray}\label{eq_09}
T^{\ast}_{1}M&=&\big\{(x,p)\in T^{\ast}M,\,g^{-1}(p,p)=1\big\}.
\end{eqnarray}
The unit normal vector field to $T^{\ast}_{1}M$ is given by
\begin{eqnarray}\label{U.N.V}
\mathcal{N}:T^{\ast}M&\rightarrow& T(T^{\ast}M)\notag\\
(x,p)&\mapsto&\mathcal{N}_{(x,p)}={}^{V}\!p.
\end{eqnarray}	

The tangential lift ${}^{T}\!\omega$ with respect to ${}^{BS}\!g$ of a covector $\omega\in T^{\ast}_{x}M$ to $(x, p)\in T^{\ast}_{1}M$ as the tangential projection of the vertical lift of $\omega$ to  $(x, p)$ with respect to $\mathcal{N}$, that is
\begin{equation*}
{}^{T}\!\omega={}^{V}\!\omega-{}^{BS}\!g_{(x,p)}({}^{V}\!\omega,\mathcal{N}_{(x,p)})\mathcal{N}_{(x,p)}={}^{V}\!\omega-g^{-1}_{x}(\omega,p){}^{V}\!p_{(x,p)}.
\end{equation*}

For the sake of notational clarity, we will use $\overline{\omega}=\omega-g^{-1}(\omega,p)p$, then ${}^{T}\!\omega={}^{V}\!\overline{\omega}$.

From the above, we get the direct sum decomposition
\begin{eqnarray}\label{D.S_2}
\qquad\qquad T_{(x,p)}T^{\ast}M=T_{(x,p)}T^{\ast}_{1}M\oplus span\{\mathcal{N}_{(x,p)}\}=T_{(x,p)}T^{\ast}_{1}M\oplus span\{{}^{V}\!p_{(x,p)}\},
\end{eqnarray}
where $(x, p)\in T^{\ast}_{1}M$.

Indeed, if $W\in T_{(x,p)}T^{\ast}M$, then they exist $X\in T_{x}M$ and $\omega\in T^{\ast}_{x}M$, such that
\begin{eqnarray}\label{eq_10}
W&=&{}^{H}\!X+{}^{V}\!\omega\notag\\
&=&{}^{H}\!X+{}^{T}\!\omega+{}^{BS}\!g_{(x,p)}({}^{V}\!\omega,\mathcal{N}_{(x,p)})\mathcal{N}_{(x,p)}\notag\\
&=&{}^{H}\!X+{}^{T}\!\omega+g^{-1}_{x}(\omega,p){}^{V}\!p_{(x,p)}.
\end{eqnarray}
From the (\ref{eq_10}) we can say that the tangent space $T_{(x,p)}T^{\ast}_{1}M$ of $T^{\ast}_{1}M$ at $(x, p)$ is given by
\begin{eqnarray*}
T_{(x,p)}T^{\ast}_{1}M= \{{}^{H}\!X+{}^{T}\!\omega\,/\, X\in T_{x}M, \omega\in \{p\}^{\bot}\subset T^{\ast}_{x}M\},
\end{eqnarray*}
where $\{p\}^{\bot}=\big\{\omega\in T^{\ast}_{x}M,\,g^{-1}(\omega,p)=0\big\}$. Hence $T_{(x,p)}T^{\ast}_{1}M$  is spanned by vectors of the form ${}^{H}\!X$ and ${}^{T}\!\omega$.

Given a covector field $\omega$ on $M$, the tangential lift ${}^{T}\!\omega$ of $\omega$ is given by
\begin{equation}\label{eq_11}
{}^{T}\!\omega_{(x,p)}=\big({}^{V}\!\omega-{}^{BS}\!g({}^{V}\!\omega,\mathcal{N})\mathcal{N}\big)_{(x,p)}={}^{V}\!\omega_{(x,p)}-g^{-1}_{x}(\omega_{x},p){}^{V}\!p_{(x,p)}.
\end{equation}

If ${}^{BS}\!\hat{g}$ is the Riemannian metric on $T^{\ast}_{1}M$ induced by ${}^{BS}\!g$, then the Levi-Civita connection ${}^{BS}\widehat{\nabla}$ of $(T^{\ast}_{1}M,{}^{BS}\!\hat{g})$ is characterized by the formula:
\begin{eqnarray}\label{C.F.U}
{}^{BS}\widehat{\nabla}_{U}V= {}^{BS}\nabla_{U}V-{}^{BS}\!g({}^{BS}\nabla_{U}V,\mathcal{N})\mathcal{N},
\end{eqnarray}
for all $U,V\in\Im^{1}_{0}(T^{\ast}M)$.

\begin{theorem}\label{th_02}
Let $(M^{n}, J, g)$ be a standard K\"{a}hler manifold and $(T^{\ast}_{1}M,{}^{BS}\!\hat{g})$ its unit cotangent bundle equipped with the Berger-type deformed Sasaki metric, then we have the following formulas:
\begin{eqnarray*}
{}^{BS}\widehat{\nabla}_{{}^{H}\!X}{}^{H}\!Y&=&{}^{H}\!(\nabla_{X}Y)+\frac{1}{2}{}^{T}\!(pR(X,Y)),\\
{}^{BS}\widehat{\nabla}_{{}^{H}\!X}{}^{T}\!\theta&=&{}^{T}\!(\nabla_{X}\theta)+\dfrac{1}{2} \big({}^{H}\!(R(\tilde{p},\tilde{\theta})X)- \delta^{2}g^{-1}(\theta,p J){}^{H}\!(R(\tilde{p},J \tilde{p})X)\big),\\
{}^{BS}\widehat{\nabla}_{{}^{T}\!\omega}{}^{H}\!Y&=&\dfrac{1}{2} \big({}^{H}\!(R(\tilde{p},\tilde{\omega})Y)- \delta^{2}g^{-1}(\omega,p J){}^{H}\!(R(\tilde{p},J \tilde{p})Y)\big),\\
{}^{BS}\widehat{\nabla}_{{}^{T}\!\omega}{}^{T}\!\theta&=&-g^{-1}(\theta, p ){}^{T}\!\omega +\delta^{2}\big(g^{-1}(\omega, p J){}^{T}\!(\theta J)+g^{-1}(\theta, p J){}^{T}\!(\omega J)\big)\\
&&-\delta^{2}\big(g^{-1}(\omega, p J)g^{-1}(\theta, p)+g^{-1}(\omega, p)g^{-1}(\theta, p J)\big){}^{T}\!(p J),
\end{eqnarray*}
for all $X,Y\in\Im^{1}_{0}(M)$ and $\omega,\theta\in\Im^{0}_{1}(M)$, where $\nabla$ is the Levi-Civita connection and $R$ is its curvature tensor.
\end{theorem}

\begin{proof} In the proof, we will use the Theorem \ref{th_01}, Lemma \ref{lem_04} and the formula (\ref{C.F.U}).

$1.$ By direct calculation, we have		
\begin{eqnarray*}
{}^{BS}\widehat{\nabla}_{{}^{H}\!X}{}^{H}\!Y&=& {}^{BS}\nabla_{{}^{H}\!X}{}^{H}\!Y-{}^{BS}\!g({}^{BS}\nabla_{{}^{H}\!X}{}^{H}\!Y,\mathcal{N})\mathcal{N}\\
&=& {}^{H}\!(\nabla_{X}Y)+\frac{1}{2}{}^{V}\!(pR(X,Y))-{}^{BS}\!g(\frac{1}{2}{}^{V}\!(pR(X,Y)),\mathcal{N})\mathcal{N}\\
&=&{}^{H}\!(\nabla_{X}Y)+\frac{1}{2}{}^{T}\!(pR(X,Y)).
\end{eqnarray*}

$2.$ We have ${}^{BS}\widehat{\nabla}_{{}^{H}\!X}{}^{T}\!\theta= {}^{BS}\nabla_{{}^{H}\!X}{}^{T}\!\theta-{}^{BS}\!g({}^{BS}\nabla_{{}^{H}\!X}{}^{T}\!\theta,\mathcal{N})\mathcal{N}$, by direct calculation, we get 
\begin{eqnarray*}
{}^{BS}\nabla_{{}^{H}\!X}{}^{T}\!\theta&=&{}^{T}\!(\nabla_{X}\theta)+\dfrac{1}{2} \big({}^{H}\!(R(\tilde{p},\tilde{\theta})X)+ \delta^{2}g^{-1}(\theta,p J){}^{H}\!(R(J \tilde{p},\tilde{p})X)\big)
\end{eqnarray*}
and
\begin{eqnarray*}
{}^{BS}\!g( {}^{BS}\nabla_{{}^{H}\!X}{}^{T}\!\theta,\mathcal{N})\mathcal{N}&=&0.
\end{eqnarray*}
Hence
\begin{eqnarray*}
{}^{BS}\widehat{\nabla}_{{}^{H}\!X}{}^{T}\!\theta &=&{}^{T}\!(\nabla_{X}\theta)+\dfrac{1}{2} \big({}^{H}\!(R(\tilde{p},\tilde{\theta})X)+ \delta^{2}g^{-1}(\theta,p J){}^{H}\!(R(J \tilde{p},\tilde{p})X)\big).
\end{eqnarray*}

$3.$ Also, we have ${}^{BS}\widehat{\nabla}_{{}^{T}\!\omega}{}^{H}\!Y= {}^{BS}\nabla_{{}^{T}\!\omega}{}^{H}\!Y-{}^{BS}\!g({}^{BS}\nabla_{{}^{T}\!\omega}{}^{H}\!Y,\mathcal{N})\mathcal{N}$,  by direct calculation, we get 
\begin{eqnarray*}
{}^{BS}\nabla_{{}^{T}\!\omega}{}^{H}\!Y&=&\dfrac{1}{2} \big({}^{H}\!(R(\tilde{p},\tilde{\omega})Y)+ \delta^{2}g^{-1}(\omega,p J){}^{H}\!(R(J \tilde{p},\tilde{p})Y)\big)
\end{eqnarray*}
and
\begin{eqnarray*}
{}^{BS}\!g( {}^{BS}\nabla_{{}^{T}\!\omega}{}^{H}\!Y,\mathcal{N})\mathcal{N}&=&0.	
\end{eqnarray*}
Hence
\begin{eqnarray*}
{}^{BS}\widehat{\nabla}_{{}^{T}\!\omega}{}^{H}\!X&=&\dfrac{1}{2} \big({}^{H}\!(R(\tilde{p},\tilde{\omega})Y)+ \delta^{2}g^{-1}(\omega,p J){}^{H}\!(R(J \tilde{p},\tilde{p})Y)\big).
\end{eqnarray*}

$4.$ In the same way above, we have ${}^{BS}\widehat{\nabla}_{{}^{T}\!\omega}{}^{T}\!\theta= {}^{BS}\nabla_{{}^{T}\!\omega}{}^{T}\!\theta-{}^{BS}\!g({}^{BS}\nabla_{{}^{T}\!\omega}{}^{T}\!\theta,\mathcal{N})\mathcal{N}$,
\begin{eqnarray*}
{}^{BS}\nabla_{{}^{T}\!\omega}{}^{T}\!\theta&=&\delta^{2}\big(g^{-1}(\omega, p J){}^{V}\!(\theta J)+g^{-1}(\theta, p J){}^{V}\!(\omega J)\big)-\delta^{2}\big(g^{-1}(\omega, p J)g^{-1}(\theta, p)\\
&&+g^{-1}(\omega, p)g^{-1}(\theta, p J)\big){}^{V}\!(p J)-g^{-1}(\theta, p){}^{V}\!\omega-g^{-1}(\omega, \theta){}^{V}\!p\\
&&+2g^{-1}(\omega, p)g^{-1}(\theta, p){}^{V}\!p,
\end{eqnarray*}
and
\begin{eqnarray*}
{}^{BS}\!g( {}^{BS}\nabla_{{}^{T}\!\omega}{}^{T}\!\theta,\mathcal{N})\mathcal{N}&=&\delta^{2}\big(g^{-1}(\omega,Jp)g^{-1}(\theta J,p)+g^{-1}(\theta,Jp)g^{-1}(\omega J,p)\big){}^{V}\!p\\
&&-g^{-1}(\omega, \theta){}^{V}\!p+g^{-1}(\omega, p)g^{-1}(\theta, p){}^{V}\!p.
\end{eqnarray*}
Hence
\begin{eqnarray*}
{}^{BS}\widehat{\nabla}_{{}^{T}\!\omega}{}^{T}\!\theta&=&-g^{-1}(\theta, p ){}^{T}\!\omega +\delta^{2}\big(g^{-1}(\omega, p J){}^{T}\!(\theta J)+g^{-1}(\theta, p J){}^{T}\!(\omega J)\big)\\
&&-\delta^{2}\big(g^{-1}(\omega, p J)g^{-1}(\theta, p)+g^{-1}(\omega, p)g^{-1}(\theta, p J)\big){}^{T}\!(p J).
\end{eqnarray*}
\end{proof}	

\section{Geodesics of the Berger-type deformed Sasaki metric}

Let  $\gamma:I \rightarrow M $ be a curve on $M$, $I$ is an open interval of $\mathbb{R}$  and $C$ be a curve on $T^{\ast}M$ expressed by  $C =(\gamma(t),\vartheta(t))$, for all $t \in I$, where $\vartheta(t)\in T^{\ast}M$ i.e. $\vartheta(t)$ is a covector field along $\gamma$. 

\begin{lemma}[\cite{Zag1}] \label{lem_05}
Let $(M,g)$ be a Riemannian manifold, and $\nabla$ denote the Levi-Civita connection of $(M,g)$. If $C =(\gamma(t),\vartheta(t))$ is a curve on $T^{\ast}M$, then
\begin{equation*}
\dot{C} = \dot{\gamma}^{H} + (\nabla_{\dot{\gamma}}\vartheta)^{V},
\end{equation*}
where $\dot{\gamma}=\frac{d\,\gamma}{d\,t}$ and $\dot{C}=\frac{d\,C}{d\,t}$.
\end{lemma}

Subsequently we denote $\gamma^{\prime}=\frac{d\,x}{d\,t}$, $\gamma^{\prime\prime}=\nabla_{\gamma^{\prime}}\gamma^{\prime}$, $\vartheta^{\prime}=\nabla_{\gamma^{\prime}}\vartheta$,  $\vartheta^{\prime\prime}=\nabla_{\gamma^{\prime}}\vartheta^{\prime}$ and $C^{\prime}=\frac{d\,C}{d\,t}$. Then
\begin{equation}\label{eq_12}
C^{\prime}={}^{H}\!\gamma^{\prime} + {}^{V}\!\vartheta^{\prime}.
\end{equation}

\begin{theorem}\label{th_03}
Let $(M^{2m}, J, g)$ be a standard K\"{a}hler manifold and $(T^{\ast}M, {}^{BS}\!g)$ its cotangent bundle equipped with the Berger-type deformed Sasaki metric. The curve $C =(\gamma(t),\vartheta(t))$ is a geodesic on $T^{\ast}M$ if and only if
\begin{eqnarray}\label{eq_13}
\left\{
\begin{array}{lll}
\gamma^{\prime\prime}=\mathcal{R}(\widetilde{\vartheta^{\prime}},\widetilde{\vartheta})\gamma^{\prime}\\
\vartheta^{\prime\prime}=2\delta^{2}g^{-1}(\vartheta^{\prime},\vartheta J)(\dfrac{\delta^{2}}{\lambda}g^{-1}(\vartheta^{\prime},\vartheta)\vartheta J-\vartheta^{\prime} J),
\end{array}
\right.
\end{eqnarray}
where $\mathcal{R}(\widetilde{\vartheta^{\prime}},\widetilde{\vartheta})=R(\widetilde{\vartheta^{\prime}},\widetilde{\vartheta})+\delta^{2}g^{-1}(\vartheta^{\prime},\vartheta J)R(\widetilde{\vartheta},J\widetilde{\vartheta})$ and $R$ is the curvature tensor of the manifold $(M^{2m}, J, g)$.
\end{theorem}

\begin{proof} 
From formula (\ref{eq_12}) and Theorem \ref{th_01}, we obtain
\begin{eqnarray*}
{}^{BS}\nabla_{C^{\prime}}C^{\prime} & = &{}^{BS}\nabla_{\displaystyle({}^{H}\!\gamma^{\prime} + {}^{V}\!\vartheta^{\prime})}({}^{H}\!\gamma^{\prime} + {}^{V}\!\vartheta^{\prime}) \\
& = &{}^{BS}\nabla_{\displaystyle{}^{H}\gamma^{\prime}}{}^{H}\gamma^{\prime} +{}^{BS}\nabla_{\displaystyle{}^{H}\!\gamma^{\prime}}{}^{V}\!\vartheta^{\prime}+{}^{BS}\nabla_{{}^{V}\!\vartheta^{\prime}}{}^{H}\!\gamma^{\prime}+{}^{BS}\nabla_{{}^{V}\!\vartheta^{\prime}}{}^{V}\!\vartheta^{\prime} \\
&=&{}^{H}\!\gamma^{\prime\prime}+{}^{H}\!(R(\widetilde{\vartheta},\widetilde{\vartheta^{\prime}})\gamma^{\prime}-\delta^{2}g^{-1}(\vartheta^{\prime},\vartheta J)R(\widetilde{\vartheta},J\widetilde{\vartheta})\gamma^{\prime}))+{}^{V}\!\vartheta^{\prime\prime}\\
&&+2\delta^{2}g^{-1}(\vartheta^{\prime},\vartheta J){}^{V}\!(\vartheta^{\prime} J)-\dfrac{2\delta^{4}}{\lambda}g^{-1}(\vartheta^{\prime},\vartheta)g^{-1}(\vartheta^{\prime},J\vartheta){}^{V}\!(\vartheta J)\\
&=&{}^{H}\!\big(\gamma^{\prime\prime}+R(\widetilde{\vartheta},\widetilde{\vartheta^{\prime}})\gamma^{\prime}-\delta^{2}g^{-1}(\vartheta^{\prime},\vartheta J)R(\widetilde{\vartheta},J\widetilde{\vartheta})\gamma^{\prime}\big)\\
&&+{}^{V}\!\big(\vartheta^{\prime\prime}+2\delta^{2}g^{-1}(\vartheta^{\prime},\vartheta J)(\vartheta^{\prime} J-\dfrac{\delta^{2}}{\lambda}g^{-1}(\vartheta^{\prime},\vartheta)\vartheta J)\big)\\
&=&{}^{H}\!\big(\gamma^{\prime\prime}-(R(\widetilde{\vartheta^{\prime}},\widetilde{\vartheta})\gamma^{\prime}+\delta^{2}g^{-1}(\vartheta^{\prime},\vartheta J)R(\widetilde{\vartheta},J\widetilde{\vartheta})\gamma^{\prime})\big)\\
&&+{}^{V}\!\big(\vartheta^{\prime\prime}-2\delta^{2}g^{-1}(\vartheta^{\prime},\vartheta J)(\dfrac{\delta^{2}}{\lambda}g^{-1}(\vartheta^{\prime},\vartheta)\vartheta J-\vartheta^{\prime} J)\big).
\end{eqnarray*}
If we put ${}^{BS}\nabla_{C^{\prime}}C^{\prime}$ equal to zero, we find $(\ref{eq_13})$.
\end{proof}

A curve $C=(\gamma(t),\vartheta(t))$ on $T^{\ast}M$ is said to be a horizontal lift of the curve $\gamma$ on $M$ if and only if $\vartheta^{\prime}=0$ \cite{Y.I}. Thus, we have

\begin{corollary}\label{co_01} Let $(M^{2m}, J, g)$ be a standard K\"{a}hler manifold and $(T^{\ast}M, {}^{BS}\!g)$ its cotangent bundle equipped with the Berger-type deformed Sasaki metric. The horizontal lift of any geodesic on $(M^{2m}, J, g)$ is a geodesic on $(T^{\ast}M, {}^{BS}\!g)$.	
\end{corollary}

\begin{corollary}\label{co_02}
Let $(M^{2m}, J, g)$ be a standard K\"{a}hler manifold and $(T^{\ast}M, {}^{BS}\!g)$ its cotangent bundle equipped with the Berger-type deformed Sasaki metric. The curve $C=(\gamma(t),\widetilde{\gamma^{\prime}(t)})$ is a geodesic on $T^{\ast}M$ if and only if $\gamma$ is a geodesic on $(M^{2m}, J, g)$.	
\end{corollary} 

\begin{proof} We have, $\gamma^{\prime}(t)\in TM $, then $\vartheta(t)=\widetilde{\gamma^{\prime}(t)}\in T^{\ast}M$. 
From $(\ref{eq_00})$ and $(\ref{eq_02})$, we get $\vartheta^{\prime}=\nabla_{\displaystyle \gamma^{\prime}}\vartheta=\widetilde{\widetilde{\nabla_{\displaystyle \gamma^{\prime}}\vartheta}}=\widetilde{\nabla_{\displaystyle \gamma^{\prime}}\widetilde{\vartheta}}=\widetilde{\nabla_{\displaystyle \gamma^{\prime}}\gamma^{\prime}}=\widetilde{\gamma^{\prime\prime}}$, then $\gamma$ is a geodesic on $M$ equivalent to $C$ is a horizontal lift of the curve $\gamma$ on $M$. Using Corollary \ref{co_01},  we deduce  the result.
\end{proof}

\begin{remark}\label{re_01}
If $\gamma$ is a geodesic on $M$	locally we have:
\begin{eqnarray*}
\gamma^{\prime\prime}=0 &\Leftrightarrow & \gamma^{\prime\prime}_{h}
+\sum_{i,j=1}^{2m}\Gamma_{ij}^{h}(\gamma^{\prime})^{i}(\gamma^{\prime})^{j}=0,\quad  h=\overline{1,2m}.  
\end{eqnarray*}
If $C$ such that $ C(t) =(\gamma(t),\vartheta(t))$ is a horizontal lift of the curve $\gamma$, locally we have: 	
\begin{eqnarray*}
\vartheta^{\prime}=0 &\Leftrightarrow & \vartheta^{\prime}_{h}-\sum_{i,j=1}^{2m}\Gamma_{jh}^{i}\vartheta_{i}(\gamma^{\prime})^{j}=0,\quad  h=\overline{1,2m}.  
\end{eqnarray*}
\end{remark}

\begin{example}
Let $\mathbb{R}^{2}$ be endowed with the structure standard K\"{a}hler $(J, g)$ defined by	
$$g= x^{2}dx^{2}+y^{2}dy^{2}.$$
and
$$ J \partial_{x}= -\frac{x}{y}\partial_{y}\quad, \quad  J \partial_{y}= \frac{y}{x}\partial_{x}.$$
The non-null Christoffel symbols of the Riemannian connection are:
$$\Gamma_{11}^{1}=\dfrac{1}{x},\; \Gamma_{22}^{2}=\dfrac{1}{y}.$$
The geodesics $\gamma$ such that $\gamma(t)=(x(t), y(t))$, $\gamma(0)= (a, b)$ and $\gamma^{\prime}(0)= (\alpha ,\beta)\in\mathbb{R}^{2}$ satisfy the system of equations,
\begin{eqnarray*}
\gamma^{\prime\prime}_{h}+\sum_{i,j=1}^{2}\Gamma_{ij}^{h}(\gamma^{\prime})^{i}(\gamma^{\prime})^{j}=0\Leftrightarrow \left\{\begin{array}{lll}
x^{\prime\prime} + \dfrac{(x^{\prime})^{2}}{x}=0\\\\
y^{\prime\prime} + \dfrac{(y^{\prime})^{2}}{y}=0
\end{array} \right.\Leftrightarrow \left\{\begin{array}{lll}
x(t)=\sqrt{2a\alpha  t+a^{2}}\\\\
y(t)=\sqrt{2b\beta t+b^{2}}
\end{array} \right.
\end{eqnarray*}
Hence $\gamma^{\prime}(t)= \dfrac{a\alpha }{\sqrt{2a\alpha  t+a^{2}}} \partial_{x}+\dfrac{b\beta}{\sqrt{2b\beta t+a^{2}}}\partial_{y}$,  $\gamma(t)= (\sqrt{2a\alpha  t+a^{2}},\sqrt{2b\beta t+b^{2}}).$

$1)$ Let $C_{1}=(\gamma(t),\vartheta(t))$ be a horizontal lift of the geodesic $\gamma$ then, 
$$\vartheta^{\prime}_{h}-\sum_{i,j=1}^{2}\Gamma_{jh}^{i}\vartheta_{i}(\gamma^{\prime})^{j}=0 \Leftrightarrow \left\{\begin{array}{lll}
\vartheta'_{1} - \dfrac{x^{\prime}}{x}\vartheta_{1}=0\\\\
\vartheta'_{2} - \dfrac{y^{\prime}}{y}\vartheta_{2}=0
\end{array} \right.\Leftrightarrow \left\{\begin{array}{lll}
\vartheta_{1}(t)=k_{1}\sqrt{2a\alpha  t+a^{2}}\\\\
\vartheta_{2}(t)=k_{2}\sqrt{2b\beta t+b^{2}}
\end{array} \right.$$
Hence $\vartheta(t)=k_{1}\sqrt{2a\alpha  t+a^{2}}dx+k_{2}\sqrt{2b\beta t+b^{2}}dy$, where  $k_{1}, k_{2}\in \mathbb{R}.$
From Corollary $\ref{co_01}$, the curve $C_{1}$ is a geodesic on $T^{\ast}\mathbb{R}^{2}$.

$2)$ Let $C_{2}=(\gamma(t),\widetilde{\gamma^{\prime}(t)})$ be a curve on $T^{\ast}\mathbb{R}^{2}$, from $(\ref{tilde_X})$, we have 
\[
    \widetilde{\gamma^{\prime}(t)}
    =\displaystyle\sum_{i,j=1}^{2} g_{ij}(\gamma^{\prime})^{j}(t)dx_{i}
    =a\alpha\sqrt{2a\alpha  t+a^{2}}dx+b\beta\sqrt{2b\beta t+b^{2}}dy.
\]
From Corollary $\ref{co_02}$, the curve $C_{2}$ is a geodesic on $T^{\ast}\mathbb{R}^{2}$.
\end{example}

\begin{corollary}\label{co_03}
Let $(M^{2m}, J, g)$ be a flat standard K\"{a}hler manifold and $(T^{\ast}M, {}^{BS}\!g)$ its cotangent bundle equipped with the Berger-type deformed Sasaki metric. Then the curve $C= (\gamma(t), \vartheta(t))$ is a geodesic on $T^{\ast}M$ if and only if $\gamma$ is a geodesic on $(M^{2m}, J, g)$ and
\begin{eqnarray*}
\vartheta^{\prime\prime}=2\delta^{2}g^{-1}(\vartheta^{\prime},\vartheta J)(\dfrac{\delta^{2}}{\lambda}g^{-1}(\vartheta^{\prime},\vartheta)\vartheta J-\vartheta^{\prime} J).
\end{eqnarray*}		
\end{corollary}

Let $C$ be a curve on $T^{\ast}M$, the cure $\gamma=\pi\circ C$ is called the projection (projected curve) of the curve $C$ on $M$. 

\begin{theorem}\label{th_04}
Let $(M^{2m}, \varphi, g)$ be a standard K\"{a}hler locally symmetric manifold, $(T^{\ast}M, ^{BS}\!\!g)$ be its  cotangent bundle equipped with the Berger-type deformed Sasaki metric, and $C$  be a geodesic on $T^{\ast}M$. Then $\mathcal{R}(\widetilde{\vartheta^{\prime}},\widetilde{\vartheta})$ is parallel along the projected curve $\gamma=\pi\circ C$.
\end{theorem}

\begin{proof} Using $(\ref{eq_03})$, $(\ref{eq_06})$ and $(\ref{eq_07})$ we have 
\begin{eqnarray*}
(\mathcal{R}(\widetilde{\vartheta^{\prime}},\widetilde{\vartheta}))^{\prime}&=&(R(\widetilde{\vartheta^{\prime}},\widetilde{\vartheta}))^{\prime}+\delta^{2}(g^{-1}(\vartheta^{\prime},\vartheta J))^{\prime} R(\widetilde{\vartheta},J\widetilde{\vartheta})+\delta^{2}g^{-1}(\vartheta^{\prime},\vartheta J) (R(\widetilde{\vartheta},J\widetilde{\vartheta}))^{\prime}\\
&=&R^{\prime}(\widetilde{\vartheta^{\prime}},\widetilde{\vartheta})+R(\widetilde{\vartheta^{\prime\prime}},\widetilde{\vartheta})+R(\widetilde{\vartheta^{\prime}},\widetilde{\vartheta^{\prime}})+\delta^{2}g^{-1}(\vartheta^{\prime\prime},\vartheta J) R(\widetilde{\vartheta},J\widetilde{\vartheta})\\
&&+\delta^{2}g^{-1}(\vartheta^{\prime},\vartheta^{\prime} J) R(\widetilde{\vartheta},J\widetilde{\vartheta})+\delta^{2}g^{-1}(\vartheta^{\prime},\vartheta J) R^{\prime}(\widetilde{\vartheta},J\widetilde{\vartheta})\\
&&+\delta^{2}g^{-1}(\vartheta^{\prime},\vartheta J) R(\widetilde{\vartheta^{\prime}},J\widetilde{\vartheta})+\delta^{2}g^{-1}(\vartheta^{\prime},\vartheta J) R(\widetilde{\vartheta},J\widetilde{\vartheta^{\prime}})\\
&=&R(\widetilde{\vartheta^{\prime\prime}},\widetilde{\vartheta})+\delta^{2}g^{-1}(\vartheta^{\prime\prime},\vartheta J) R(\widetilde{\vartheta},J\widetilde{\vartheta})+\delta^{2}g^{-1}(\vartheta^{\prime},\vartheta J) R(\widetilde{\vartheta^{\prime}},J\widetilde{\vartheta})\\
&&+\delta^{2}g^{-1}(\vartheta^{\prime},\vartheta J) R(\widetilde{\vartheta},J\widetilde{\vartheta^{\prime}}),
\end{eqnarray*}
from second equation of $(\ref{eq_13})$ and $(\ref{eq_08})$ we get
\begin{eqnarray*}
(\mathcal{R}(\widetilde{\vartheta^{\prime}},\widetilde{\vartheta}))^{\prime}&=&\dfrac{2\delta^{4}}{\lambda}g^{-1}(\vartheta^{\prime},\vartheta J)g^{-1}(\vartheta^{\prime},\vartheta)R(\widetilde{\vartheta J},\widetilde{\vartheta})-2\delta^{2}g^{-1}(\vartheta^{\prime},\vartheta J) R(\widetilde{\vartheta^{\prime}J},\widetilde{\vartheta})\\
&&+\dfrac{2\delta^{6}}{\lambda}g^{-1}(\vartheta^{\prime},\vartheta J)g^{-1}(\vartheta^{\prime},\vartheta)g^{-1}(\vartheta J,\vartheta J)R(\widetilde{\vartheta },J\widetilde{\vartheta})\\
&&-2\delta^{4}g^{-1}(\vartheta^{\prime},\vartheta J)g^{-1}(\vartheta^{\prime} J,\vartheta J)R(\widetilde{\vartheta },J\widetilde{\vartheta})+2\delta^{2}g^{-1}(\vartheta^{\prime},\vartheta J) R(\widetilde{\vartheta^{\prime}},J\widetilde{\vartheta})\\
&=&\dfrac{2\delta^{4}}{\lambda}g^{-1}(\vartheta^{\prime},\vartheta J)g^{-1}(\vartheta^{\prime},\vartheta)R(\widetilde{\vartheta },J\widetilde{\vartheta})\\
&&+\dfrac{2\delta^{4}(\lambda-1)}{\lambda}g^{-1}(\vartheta^{\prime},\vartheta J)g^{-1}(\vartheta^{\prime},\vartheta)R(\widetilde{\vartheta },J\widetilde{\vartheta})\\
&&-2\delta^{4}g^{-1}(\vartheta^{\prime},\vartheta J)g^{-1}(\vartheta^{\prime},\vartheta )R(\widetilde{\vartheta },J\widetilde{\vartheta})\\
&=&(\dfrac{2\delta^{4}}{\lambda}+\dfrac{2\delta^{4}(\lambda-1)}{\lambda}-2\delta^{4})g^{-1}(\vartheta^{\prime},\vartheta J)g^{-1}(\vartheta^{\prime},\vartheta )R(\widetilde{\vartheta },J\widetilde{\vartheta})\\
&=&0.
\end{eqnarray*}	
\end{proof}

We now study the geodesics  on the unit cotangent bundle with respect to the Berger-type deformed Sasaki metric.

\begin{lemma}\label{lem_06}
Let $(M^{2m}, \varphi, g)$ be a standard K\"{a}hler manifold, $(T^{\ast}_{1}M,{}^{BS}\!\hat{g})$ its unit cotangent bundle equipped with the Berger-type deformed Sasaki metric and $C=(\gamma(t),\vartheta(t))$  be a curve on $T^{\ast}_{1}M$. Then we have
\begin{equation}\label{eq_14}
C^{\prime} = {}^{H}\!\gamma^{\prime} + {}^{T}\!\vartheta^{\prime}.
\end{equation}
\end{lemma}

\begin{proof} Using $(\ref{eq_12})$, we have 
\begin{eqnarray*}
C^{\prime} &=& {}^{H}\!\gamma^{\prime} + {}^{V}\!\vartheta^{\prime}={}^{H}\!\gamma^{\prime} + {}^{T}\!\vartheta^{\prime}+g^{-1}(\vartheta^{\prime},\vartheta){}^{V}\!\vartheta.
\end{eqnarray*}	 
Since $C(t) =(\gamma(t),\vartheta(t))\in T^{\ast}_{1}M$ then $g^{-1}(\vartheta,\vartheta)=1$, on the other hand
\begin{eqnarray*}
0&=&(g^{-1}(\vartheta,\vartheta))^{\prime}=2g^{-1}(\vartheta^{\prime},\vartheta),
\end{eqnarray*}
hence
\begin{eqnarray}\label{eq_15}
\quad g^{-1}(\vartheta^{\prime},\vartheta)=0.
\end{eqnarray}
The proof of the lemma is completed.
\end{proof}
Subsequently, let $t$ be an arc length parameter on $C$, From \ref{eq_14}, we have
\begin{eqnarray}\label{eq_16}
1=|\gamma^{\prime}|^{2}+ |\vartheta^{\prime}|^{2}+\delta^{2}g^{-1}(\vartheta^{\prime},\vartheta J)^{2}.
\end{eqnarray}

\begin{theorem}\label{th_05}
Let $(M^{2m}, \varphi, g)$ be a standard K\"{a}hler manifold, $(T^{\ast}_{1}M,{}^{BS}\!\hat{g})$ its unit cotangent bundle equipped with the Berger-type deformed Sasaki metric and $C=(\gamma(t),\vartheta(t))$  be a curve on $T^{\ast}_{1}M$. Let $\kappa=|\vartheta^{\prime}|$ and $\mu=g^{-1}(\vartheta^{\prime},\vartheta J)$. Then $C$ is a geodesic on $T^{\ast}_{1}M$ if and only if
\begin{eqnarray}\label{eq_17}
\left\{
\begin{array}{lll}
\gamma^{\prime\prime}=\mathcal{R}(\widetilde{\vartheta^{\prime}},\widetilde{\vartheta})\gamma^{\prime}\\
\vartheta^{\prime\prime}=-2\delta^{2}\mu\,\vartheta^{\prime}J,
\end{array}
\right.
\end{eqnarray}
where $\mathcal{R}(\widetilde{\vartheta^{\prime}},\widetilde{\vartheta})=R(\widetilde{\vartheta^{\prime}},\widetilde{\vartheta})+\delta^{2}\mu R(\widetilde{\vartheta},J\widetilde{\vartheta})$. Moreover, 
\begin{equation}\label{eq_18}
\left\{
\begin{array}{lll}
|\vartheta^{\prime}|=\kappa\\
|\gamma^{\prime}|=\sqrt{1-K}
\end{array}
\right.
\end{equation}
where $K =\kappa^{2}+\delta^{2}\mu^{2} =const$, $0 \leq K \leq 1$, $\kappa=const$ and $\mu=const$.
\end{theorem}

\begin{proof} Using formula $(\ref{eq_14})$ and Theorem \ref{th_02}, we compute the derivative $\widehat{\nabla}_{C^{\prime}}C^{\prime}$. 
\begin{eqnarray*}
\widehat{\nabla}_{C^{\prime}}C^{\prime} &=&\widehat{\nabla}_{\displaystyle({}^{H}\!\gamma^{\prime} + {}^{T}\!\vartheta^{\prime})}({}^{H}\!\gamma^{\prime} + {}^{T}\!\vartheta^{\prime}) \\
&=&\widehat{\nabla}_{\displaystyle{}^{H}\gamma^{\prime}}{}^{H}\gamma^{\prime} +\widehat{\nabla}_{\displaystyle{}^{H}\!\gamma^{\prime}}{}^{T}\!\vartheta^{\prime}+\widehat{\nabla}_{{}^{T}\!\vartheta^{\prime}}{}^{H}\!\gamma^{\prime}+\widehat{\nabla}_{{}^{T}\!\vartheta^{\prime}}{}^{T}\!\vartheta^{\prime} \\
&=&{}^{H}\!\gamma^{\prime\prime}+{}^{T}\!\vartheta^{\prime\prime}+{}^{H}\!(R(\widetilde{\vartheta},\widetilde{\vartheta^{\prime}})\gamma^{\prime}-\delta^{2}g^{-1}(\vartheta^{\prime},\vartheta J)R(\widetilde{\vartheta},J\widetilde{\vartheta})
\gamma^{\prime}))\\
&&+2\delta^{2}g^{-1}(\vartheta^{\prime},\vartheta J){}^{T}\!(\vartheta^{\prime}J)\\
&=&{}^{H}\!\gamma^{\prime\prime}-{}^{H}\!(R(\widetilde{\vartheta^{\prime}},\widetilde{\vartheta})\gamma^{\prime}+\delta^{2}g^{-1}(\vartheta^{\prime},\vartheta J)R(\widetilde{\vartheta},J\widetilde{\vartheta})
\gamma^{\prime}))\\
&&+{}^{T}\!\vartheta^{\prime\prime}+2\delta^{2}g^{-1}(\vartheta^{\prime},\vartheta J){}^{T}\!(\vartheta^{\prime}J)\\
&=&{}^{H}\!\big(\gamma^{\prime\prime}-\mathcal{R}(\widetilde{\vartheta^{\prime}},\widetilde{\vartheta})\gamma^{\prime}\big)+{}^{T}\!\big(\vartheta^{\prime\prime}+2\delta^{2}\mu\,\vartheta^{\prime}J\big).
\end{eqnarray*}
If we put $\widehat{\nabla}_{C^{\prime}}C^{\prime}$ equal to zero, we find $(\ref{eq_17})$.
Moreover, we have $\kappa=|\vartheta^{\prime}|$, then \[(\kappa^{2})^{\prime}=2g^{-1}(\vartheta^{\prime\prime},\vartheta^{\prime}),\] from second equation of $(\ref{eq_17})$, we have 
$g^{-1}(\vartheta^{\prime\prime},\vartheta^{\prime})+2\delta^{2}\mu g^{-1}(\vartheta^{\prime}J,\vartheta^{\prime})=0$, on the other hand, from $(\ref{Herm-m1})$, we find $g^{-1}(\vartheta^{\prime}J,\vartheta^{\prime})=0$, then  $g^{-1}(\vartheta^{\prime\prime},\vartheta^{\prime})=0$, hence $\kappa= const$.
We have, $\mu=g^{-1}(\vartheta^{\prime},\vartheta J)$, then $\mu^{\prime}=g^{-1}(\vartheta^{\prime\prime},\vartheta J)+g^{-1}(\vartheta^{\prime},\vartheta^{\prime} J)=g^{-1}(\vartheta^{\prime\prime},\vartheta J)$, from second equation of $(\ref{eq_17})$, we have $\mu^{\prime}=g^{-1}(\vartheta^{\prime\prime},\vartheta J)=2\delta^{2}\mu g^{-1}(\vartheta^{\prime}J,\vartheta J)$, from $(\ref{Herm-m1})$, we find \[g^{-1}(\vartheta^{\prime}J,\vartheta J)=g^{-1}(\vartheta^{\prime},\vartheta )=0, \] hence $\mu= const$.
Using $(\ref{eq_16})$, we get $	1=|\gamma^{\prime}|^{2}+ \kappa^{2}+\delta^{2}\mu^{2}$, then
\begin{eqnarray*}
|\gamma^{\prime}|=\sqrt{1-(\kappa^{2}+\delta^{2}\mu^{2})}=\sqrt{1-K}.
\end{eqnarray*}
where, $K =\kappa^{2}+\delta^{2}\mu^{2} =const.$
\end{proof}

\begin{theorem}\label{th_06}
Let $(M^{2m}, J, g)$ denote a standard Kähler locally symmetric manifold, $(T^{\ast}_{1}M, ^{BS}\!\hat{g})$ be its unit cotangent bundle equipped with the Berger-type deformed Sasaki metric, and $C$ be a geodesic on $T^{\ast}_{1}M$. Then $\mathcal{R}(\widetilde{\vartheta^{\prime}},\widetilde{\vartheta})$ is parallel along the projected curve $\gamma=\pi\circ C$.
\end{theorem}

\begin{proof} 
Similarly, proving Theorem $\ref{th_04}$, using $\mu=g^{-1}(\vartheta^{\prime},\vartheta J)$ and Theorem $\ref{th_05}$, we get the result.
\end{proof}

\begin{theorem}\label{th_07}
Let $(M^{2m}, J, g)$ denote a standard K\"{a}hler locally symmetric manifold, $(T^{\ast}_{1}M,{}^{BS}\!\hat{g})$ be its unit tangent bundle equipped with Berger-type deformed Sasaki metric, and $C$ be a geodesic on $T^{\ast}_{1}M$, then all Frenet curvatures of the projected curve $\gamma=\pi\circ C$ are constants.
\end{theorem}

\begin{proof} Using the first equation of $(\ref{eq_17})$, we have $$\gamma^{\prime\prime}=\mathcal{R}(\widetilde{\vartheta^{\prime}},\widetilde{\vartheta})\gamma^{\prime}=R(\widetilde{\vartheta^{\prime}},\widetilde{\vartheta})\gamma^{\prime}+\delta^{2}\mu R(\widetilde{\vartheta},J\widetilde{\vartheta})\gamma^{\prime}.$$
Since $(g(\gamma^{\prime}, \gamma^{\prime}))^{\prime}=2g(\gamma^{\prime\prime}, \gamma^{\prime})=2g(\mathcal{R}(\widetilde{\vartheta^{\prime}},\widetilde{\vartheta})\gamma^{\prime}, \gamma^{\prime})=0$, hence $|\gamma^{\prime}|= const$.
\begin{eqnarray*}
\gamma^{\prime\prime\prime}&=&(\mathcal{R}(\widetilde{\vartheta^{\prime}},\widetilde{\vartheta})\gamma^{\prime})^{\prime}\\
&=&(R(\widetilde{\vartheta^{\prime}},\widetilde{\vartheta})\gamma^{\prime})^{\prime}+\delta^{2}\mu (R(\widetilde{\vartheta},J\widetilde{\vartheta})\gamma^{\prime})^{\prime}\\
&=&R^{\prime}(\widetilde{\vartheta^{\prime}},\widetilde{\vartheta})\gamma^{\prime}+R(\widetilde{\vartheta^{\prime\prime}},\widetilde{\vartheta})\gamma^{\prime}+R(\widetilde{\vartheta^{\prime}},\widetilde{\vartheta^{\prime}})\gamma^{\prime}+R(\widetilde{\vartheta^{\prime}},\widetilde{\vartheta})\gamma^{\prime\prime}\\
&&+\delta^{2}\mu(R^{\prime}(\widetilde{\vartheta},J\widetilde{\vartheta})\gamma^{\prime}+R(\widetilde{\vartheta^{\prime}},J\widetilde{\vartheta})\gamma^{\prime}+R(\widetilde{\vartheta},J\widetilde{\vartheta^{\prime}})\gamma^{\prime}+R(\widetilde{\vartheta},J\widetilde{\vartheta})\gamma^{\prime\prime})\\
&=&R(\widetilde{\vartheta^{\prime\prime}},\widetilde{\vartheta})\gamma^{\prime}+R(\widetilde{\vartheta^{\prime}},\widetilde{\vartheta})\gamma^{\prime\prime}+\delta^{2}\mu(R(\widetilde{\vartheta^{\prime}},J\widetilde{\vartheta})\gamma^{\prime}+R(\widetilde{\vartheta},J\widetilde{\vartheta^{\prime}})\gamma^{\prime}+R(\widetilde{\vartheta},J\widetilde{\vartheta})\gamma^{\prime\prime})\\
&=&R(\widetilde{\vartheta^{\prime\prime}},\widetilde{\vartheta})\gamma^{\prime}-2\delta^{2}\mu\,R(J\widetilde{\vartheta^{\prime}},\widetilde{\vartheta})\gamma^{\prime}+\mathcal{R}(\widetilde{\vartheta^{\prime}},\widetilde{\vartheta})\gamma^{\prime\prime}\\
&=&R(\widetilde{\vartheta^{\prime\prime}},\widetilde{\vartheta})\gamma^{\prime}+2\delta^{2}\mu\,R(\widetilde{\vartheta^{\prime}J},\widetilde{\vartheta})\gamma^{\prime}+\mathcal{R}(\widetilde{\vartheta^{\prime}},\widetilde{\vartheta})\gamma^{\prime\prime}\\
&=&R(\widetilde{\vartheta^{\prime\prime}},\widetilde{\vartheta})\gamma^{\prime}-R(\widetilde{\vartheta^{\prime\prime}},\widetilde{\vartheta})\gamma^{\prime}+\mathcal{R}(\widetilde{\vartheta^{\prime}},\widetilde{\vartheta})\gamma^{\prime\prime}\\
&=&\mathcal{R}(\widetilde{\vartheta^{\prime}},\widetilde{\vartheta})\gamma^{\prime\prime}.
\end{eqnarray*}
Since $(g(\gamma^{\prime\prime}, \gamma^{\prime\prime}))^{\prime}=2g(x^{\prime\prime\prime}, \gamma^{\prime\prime})=2g(\mathcal{R}(\widetilde{\vartheta^{\prime}},\widetilde{\vartheta})\gamma^{\prime\prime}, \gamma^{\prime\prime})=0$, hence $|\gamma^{\prime\prime}|= const$.

Continuing the process by recurrence, we get 
\begin{equation*}
\gamma^{(p+1)}=\mathcal{R}(\widetilde{\vartheta^{\prime}},\widetilde{\vartheta})\gamma^{(p)},\quad p\geq1
\end{equation*}
and
\begin{equation*}
(g(\gamma^{(p)}, \gamma^{(p)}))^{\prime}=2g(\gamma^{(p+1)},\gamma^{(p)})=2g(\mathcal{R}(\widetilde{\vartheta^{\prime}},\widetilde{\vartheta})\gamma^{(p)},\gamma^{(p)})=0.
\end{equation*}
Thus, we get 
\begin{equation}\label{eq_19}
|\gamma^{(p)}|=const,\quad p\geq 1.  
\end{equation}
Denote by $s$ an arc length parameter on $\gamma $, i.e. $(|\gamma^{\prime}_{s}|=1)$. Then $\gamma^{\prime}=
\gamma^{\prime}_{s}\frac{d\,s}{dt}$, and using $(\ref{eq_18})$, we get
\begin{equation}\label{eq_20}
\frac{d\,s}{dt}=\sqrt{1-K}=const.	
\end{equation}
Let $\nu_{1} = \gamma^{\prime}_{s}$ be the first vector in the Frenet frame $\nu_{1}, \ldots, \nu_{2m-1}$ along $\gamma$ and let $k_{1}, \ldots, k_{2m-1}$ the Frenet curvatures of $\gamma$. Then the Frenet formulas verify
\begin{equation}\label{F.F}
\left\{
\begin{array}{lll}
(\nu_{1})^{\prime}_{s}&=&k_{1}\nu_{2}\\
(\nu_{i})^{\prime}_{s}&=&-k_{i-1}\nu_{i-1}+k_{i}\nu_{i+1},\quad 2\leq i\leq 2m-2\\
(\nu_{2m-1})^{\prime}_{s}&=&-k_{2m-2}\nu_{2m-2}
\end{array}
\right.
\end{equation}
Using $(\ref{eq_20})$ and the Frenet formulas $(\ref{F.F})$, we obtain
$$\gamma^{\prime}=\gamma^{\prime}_{s}\frac{d\,s}{dt}=\sqrt{1-K}\,\nu_{1}.$$
$$\gamma^{\prime\prime}=\sqrt{1-K}(\nu_{1})^{\prime}_{t}=\sqrt{1-K}(\nu_{1})^{\prime}_{s}\frac{d\,s}{dt}=(1-K)k_{1}\nu_{2}.$$
Now $(\ref{eq_19})$ implies $k_{1} = const$. Next, in a similar way, we have
$$\gamma^{\prime\prime\prime}=(1-K)k_{1}(\nu_{2})^{\prime}_{t}=(1-K)k_{1}(\nu_{2})^{\prime}_{s}\frac{d\,s}{dt}=(1-K)\sqrt{1-K}k_{1}(-k_{1}\nu_{1}+k_{2}\nu_{3}).$$
and again $(\ref{eq_19})$ implies $k_{2}=const$. By continuing the process, we finish the proof.
\end{proof}

\begin{lemma}\label{lem_07}
Let $(M^{2m}, J, g)$ be a standard K\"{a}hler manifold, $(T^{\ast}_{1}M,{}^{BS}\!\hat{g})$ its unit cotangent bundle equipped with Berger-type deformed Sasaki metric and  $C=(\gamma(t),\vartheta(t))$ be a curve on $T^{\ast}_{1}M$, we put $\xi=\vartheta J$, then we have
\begin{enumerate}
  \item  $\Gamma=(\gamma(t),\xi(t))$ is a curve on $T^{\ast}_{1}M$. 
\item  $\Gamma$ is a geodesic on $T^{\ast}_{1}M$ if and only if $C$ is a geodesic on $T^{\ast}_{1}M$.
\end{enumerate}
\end{lemma}

\begin{proof}	
\begin{enumerate}
\item Since we have $\xi(t)=\vartheta J(t)$, then $g^{-1}(\xi,\xi)=g^{-1}(\vartheta J,\vartheta J)=g^{-1}(\vartheta,\vartheta)$.  Since $C =(\gamma(t),\vartheta(t))\in T^{\ast}_{1}M$ we get  $g(\vartheta,\vartheta)=1$.  Hence, $g(\xi,\xi)=1$, which means that $\Gamma =(\gamma(t),\xi(t))\in T^{\ast}_{1}M$. 

\item  In a similar way proof of $(\ref{eq_17})$, we have
\begin{eqnarray*}
\widehat{\nabla}_{\Gamma^{\prime}}\Gamma^{\prime} &=& {}^{H}\!\big(\gamma^{\prime\prime}-\mathcal{R}(\widetilde{\xi^{\prime}},\widetilde{\xi})\gamma^{\prime}\big)+{}^{T}\!\big(\xi^{\prime\prime}+2\delta^{2}\mu\,\xi^{\prime}J\big),
\end{eqnarray*}
since  $\xi^{\prime}=\vartheta^{\prime} J$, $\xi^{\prime\prime}=\vartheta^{\prime\prime}J$ and $\mathcal{R}(\widetilde{\xi^{\prime}},\widetilde{\xi})=\mathcal{R}(\widetilde{\vartheta^{\prime}},\widetilde{\vartheta})$, we have
\begin{eqnarray*}
\widehat{\nabla}_{\Gamma^{\prime}}\Gamma^{\prime} &=&{}^{H}\!\big(\gamma^{\prime\prime}-\mathcal{R}(\widetilde{\vartheta^{\prime}},\widetilde{\vartheta})\gamma^{\prime}\big)+{}^{T}\!\big((\vartheta^{\prime\prime}+2\delta^{2}\mu\,\vartheta^{\prime}J)J\big).
\end{eqnarray*}
\begin{equation*}
\widehat{\nabla}_{\Gamma^{\prime}}\Gamma^{\prime}=0 \Leftrightarrow	\left\{
\begin{array}{lll}
\gamma^{\prime\prime}-\mathcal{R}(\widetilde{\vartheta^{\prime}},\widetilde{\vartheta})\gamma^{\prime}=0\\
\vartheta^{\prime\prime}+2\delta^{2}\mu\,\vartheta^{\prime}J=0
\end{array}
\right.\Leftrightarrow \left\{
\begin{array}{lll}
\gamma^{\prime\prime}=\mathcal{R}(\widetilde{\vartheta^{\prime}},\widetilde{\vartheta})\gamma^{\prime}\\
\vartheta^{\prime\prime}=-2\delta^{2}\mu\,\vartheta^{\prime}J
\end{array}
\right.\Leftrightarrow\widehat{\nabla}_{C^{\prime}}C^{\prime}=0.
\qedhere
\end{equation*}
\end{enumerate}
\end{proof}

From Theorem \ref{th_07} and Lemma \ref{lem_07}, we have the following theorem: 

\begin{theorem}\label{th_08}
Let $(M^{2m}, J, g)$ denote a standard K\"{a}hler locally symmetric manifold, $(T^{\ast}_{1}M,{}^{BS}\!\hat{g})$ be its unit cotangent bundle equipped with Berger-type deformed Sasaki metric, and $C=(\gamma(t),\vartheta(t))$ be a geodesic on $T^{\ast}_{1}M$, we put $\xi=\vartheta J$, then all Frenet curvatures  of the projected curve $\gamma=\pi\circ \Gamma$ are constants, where $\Gamma=(\gamma(t), \xi(t))$.
\end{theorem}

\subsection*{Acknowledgments}
The author expresses his gratitude to the anonymous referee for his valuable comments and suggestions towards improving the paper. The author would also like to thank Professor  Ahmed Mohammed Cherif, University Mustapha Stambouli of Mascara Algeria, for her helpful suggestions and valuable comments.
This research was supported by the National Algerian P.R.F.U. project.


\EditInfo{March 3, 2023}{May 9, 2023}{Ilka Agricola}

\end{document}